\documentclass[a4paper,12pt]{article} 
\title{Multiplicativity of the double ramification cycle}

\usepackage{amsmath, amssymb, mathrsfs, amsthm, geometry}
\usepackage{mathtools} 
\usepackage{hyperref}
\usepackage[all]{xy}

\usepackage{enumitem}

\usepackage{cleveref}
\let\oref\ref
\AtBeginDocument{\renewcommand{\ref}[1]{\cref{#1}}}

\usepackage{pgf,tikz}
\usetikzlibrary{arrows}

\usepackage{tikz-cd} 

\newcommand{\on}[1]{\operatorname{#1}}
\newcommand{\bb}[1]{{\mathbb{#1}}}
\newcommand{\cl}[1]{{\mathscr{#1}}}
\newcommand{\ca}[1]{{\mathcal{#1}}}

\newcommand{\ul}[1]{{\underline{#1}}}

\newcommand{\hra}{\hookrightarrow}
\newcommand{\sub}{\subseteq}
\newcommand{\tra}{\rightarrowtail}

\theoremstyle{definition}
\newtheorem{definition}{Definition}[section]
\newtheorem{conjecture}[definition]{Conjecture}

\theoremstyle{plain}
\newtheorem{proposition}[definition]{Proposition}
\newtheorem{lemma}[definition]{Lemma}
\newtheorem{theorem}[definition]{Theorem}

\theoremstyle{remark}

\newtheorem{remark}[definition]{Remark}

\renewcommand{\phi}{\varphi}

\author{David Holmes, Aaron Pixton, Johannes Schmitt}
\date{\today}

\newcounter{nootje}
\newcommand{\expanded}[1]{#1}
\renewcommand{\expanded}[1]{}

\newcommand{\beq}{\begin{equation}}
\newcommand{\eeq}{\end{equation}}
\newcommand{\beqs}{\begin{equation*}}
\newcommand{\eeqs}{\end{equation*}}

\begin{document}
\maketitle
\begin{abstract} 
The double ramification cycle satisfies a basic multiplicative relation $\on{DRC}_a \cdot \on{DRC}_b = \on{DRC}_a \cdot \on{DRC}_{a + b}$ over the locus of compact-type curves, but this relation fails in the Chow ring of the moduli space of stable curves. We restore this relation over the moduli space of stable curves by introducing an extension of the double ramification cycle to the small b-Chow ring (the colimit of the Chow rings of all smooth blowups of the moduli space). We use this to give evidence for the conjectured equality between the (twisted) double ramification cycle and a cycle ${\operatorname P}_g^{d,k}({A})$ described by the second author in \cite{Janda2016Double-ramifica}. 
\end{abstract}


\tableofcontents

\newcommand{\Mtildes}{ \widetilde{\ca M}^\Sigma}
\newcommand{\Mhat}{\hat{\cl M}}
\newcommand{\sch}[1]{\textcolor{blue}{#1}}

\newcommand{\Mbar}{\overline{\ca M}}
\newcommand{\MD}{\ca M^\blacklozenge}
\newcommand{\Md}{\ca M^\lozenge}

\section{Introduction}\label{sec:intro}

Given integers $a_1, \dots, a_n$ summing to zero, one defines the \emph{double ramification cycle} $\on{DRC}_{\ul a}$ in the moduli space $\ca M_{g,n}$ of smooth curves by pulling back the unit section of the universal jacobian along the section induced by the divisor $\sum_i a_i [x_i]$, where the $x_i$ are the tautological sections of the universal curve. This class has been extended over the whole of $\Mbar_{g,n}$ by work of Li-Graber-Vakil \cite{Li2001Stable-morphism}, \cite{Li2002A-degeneration-}, \cite{Graber2005Relative-virtua} (extending work of Hain \cite{Hain2013Normal-function} and Grushevsky-Zakharov \cite{Grushevsky2012The-zero-sectio}). An alternative construction of the same cycle was recently given by the first author \cite{Holmes2017Extending-the-d} (and another by Kass-Pagani \cite{Kass2017The-stability-s}, but we will not use the latter in this paper). 

A basic multiplicative relation holds between the double ramification cycles over the locus of curves of compact-type, namely
\begin{equation}\label{eq:relation_intro}
\on{DRC}_{\ul a} \cdot \on{DRC}_{\ul b} = \on{DRC}_{\ul a} \cdot \on{DRC}_{\ul a + \ul b}
\end{equation}
for all vectors $\ul a$, $\ul b$ of ramification data. In \ref{sec:example} we show by means of an example that this relation fails to hold in the Chow ring of $\Mbar_{g,n}$, and moreover that this cannot be corrected by making a different choice of extension of the cycle. 

The aim of this paper is to restore the relation (\oref{eq:relation_intro}) over the whole of $\Mbar_{g,n}$ by working in the (small) $b$-Chow ring $\on{bCH}_{\bb Q}(\Mbar_{g,n})$, defined as the colimit of the Chow rings of all smooth blowups of $\Mbar_{g,n}$ (see \ref{sec:b_Chow}). The transition maps are given by pullback of cycles; the relation to Shokurov's notion of $b$-divisor (\cite{Shokurov19963-fold-log-mode}, \cite{Shokurov2003Prelimiting-fli}) is discussed further in \ref{sec:b_Chow}. Using results of \cite{Holmes2017Extending-the-d}, we construct extensions $\on{bDRC}_{\ul a}$ of the double ramification cycle in the small $b$-Chow ring $\on{bCH}_{\bb Q}(\Mbar_{g,n})$ with two fundamental properties: 

\begin{theorem}\label{thm:pushforward}
The pushforward of $\on{bDRC}_{\ul a}$ to the Chow ring of $\Mbar_{g,n}$ coincides with the standard extension of the double ramification cycle $\overline{\on{DRC}}_{\ul a}$ (as constructed in \cite{Li2001Stable-morphism}, \cite{Li2002A-degeneration-}, and \cite{Graber2005Relative-virtua}, or equivalently in \cite{Holmes2017Extending-the-d}). 
\end{theorem}

\begin{theorem}\label{thm:mult_intro}
The relation $\on{bDRC}_{\ul a} \cdot \on{bDRC}_{\ul b} = \on{bDRC}_{\ul a} \cdot \on{bDRC}_{\ul a + \ul b}
$ holds in the small $b$-Chow ring $\on{bCH}_{\bb Q}(\Mbar_{g,n})$. 
\end{theorem}
This result holds also for the $\omega^{\otimes k}$-twisted version of the double ramification cycle, with essentially the same proof. 

Note that the pushforward map from small $b$-Chow ring $\on{bCH}_{\bb Q}(\Mbar_{g,n})$ to the Chow ring $\on{CH}_{\bb Q}(\Mbar_{g,n})$ is \emph{not} a ring homomorphism, so these results do \emph{not} imply multiplicativity of the $\overline{\on{DRC}}$ in $\on{CH}_{\bb Q}(\Mbar_{g,n})$. 

The relation (\oref{eq:relation_intro}) is extremely natural, and we might speculate that its failure to hold in the Chow group of $\Mbar_{g,n}$ suggests that this is not the most natural setting in which to consider the double ramification cycle. Perhaps the $b$-Chow version of the double ramification cycle is the more fundamental object, or at least a shadow thereof? 

Conjecture 1.4 of \cite{Holmes2017Extending-the-d} predicts that the cycle $\overline{\on{DRC}}_{\ul a}$ in $\on{CH}_{\bb Q}(\Mbar_{g,n})$ coincides with a cycle $2^{-g}\mathrm{P}_g^{g,k}({\ul A})$ constructed by the second named author; more details are given in \ref{sec:relation_to_P}. For $k=0$ this follows from the main theorem of \cite{Janda2016Double-ramifica}, but it is open for higher $k$. In \ref{prop:pixton_eq_ct} we verify this conjecture on the locus of compact-type curves. 

In \ref{sec:treelike} we show that the multiplicativity relation \ref{eq:relation_intro} holds in the Chow ring of the locus of \emph{treelike} curves --- curves whose dual graph has cycles of length at most $1$. In particular, if the conjectured equality between $\overline{\on{DRC}}_{\ul a}$ and $2^{-g}\mathrm{P}_g^{g,k}({\ul A})$ holds true, then in turn the cycle $\mathrm{P}_g^{g,k}({\ul A})$ must also satisfy this multiplicativity relation on the locus of treelike curves. In \ref{prop:pixton_mult} we give a direct, combinatorial proof of this multiplicativity relation for $\mathrm{P}_g^{g,k}({\ul A})$, providing evidence for the conjectural equality between $\overline{\on{DRC}}_{\ul a}$ and $2^{-g}\mathrm{P}_g^{g,k}({\ul A})$.
\subsection*{Acknowledgements}

Part of this work was carried out at ETH Z\"urich, where the first author was kindly invited by Rahul Pandharipande. The second author was supported by a fellowship from the Clay Mathematics Institute. The third author wants to thank Felix Janda and Samuel Grushevsky for very useful advice. The third author was supported by the grant SNF-200020162928.

\subsection*{Notation and setup}

We write $\hra$ for open immersions and $\tra$ for closed immersions. We work over a field of characteristic zero, so that we can assume resolution of singularities. See \ref{sec:log_version} for an approach that works in arbitrary characteristic. 

For us, `curve' means proper, flat, finitely presented, with reduced connected nodal geometric fibres, and $\Mbar_{g,n}$ denotes the usual Deligne-Mumford-Knudsen compactification of the moduli stack of smooth curves of genus g with $n$ disjoint ordered marked sections. We write $\ca C_{g,n}/\Mbar_{g,n}$ for the universal curve, $x_i$ for the sections, and $\omega$ for the relative dualising sheaf. We let $\ca J_{g,n} = \on{Pic}^0_{\ca C_{g,n}/\Mbar_{g,n}}$ denote the universal jacobian (a semiabelian scheme, the fibrewise connected component of the identity in $ \on{Pic}_{\ca C/\Mbar}$).

\section{Extending the double ramification cycle}\label{sec:extending_DRC}

Here we recall briefly the construction of the extension of the double ramification cycle given in \cite{Holmes2017Extending-the-d}. Given integers $\ul a = (a_1, \dots, a_n, k)$ with $\sum_i a_i = k(2g-2)$, we define a section $\sigma_{\ul a} = [\omega^{\otimes k}\left(-\sum_i a_i x_i\right)]$ of $\ca J_{g,n}$ over $\ca M_{g,n}$ (which does not in general extend over the whole of $\Mbar_{g,n}$). 

 Let $f\colon X \to \Mbar_{g,n}$ be a proper birational map from a regular stack (a `regular modification'). The section $\sigma_{\ul a}$ is then defined on some dense open of $X$. We write $\mathring X$ for the largest open of $X$ on which this rational map can be extended to a morphism, and $\sigma_{\ul a}^X\colon \mathring X \to J$ for the extension. 

We define the \emph{double ramification locus} $\on{DRL}_{\ul a}^X \tra \mathring X$ to be the schematic pullback of the unit section of $\ca J_{g,n}$ along $\sigma_{\ul a}^X$, and the \emph{double ramification cycle} $\on{DRC}_{\ul a}^X$ to be the cycle-theoretic pullback, as a cycle supported on $\on{DRL}_{\ul a}^X$. Now the map $\mathring X \to \Mbar_{g,n}$ is rarely proper, but we have: 

\begin{theorem}[\cite{Holmes2017Extending-the-d}, theorem 1.1]\label{main_theorem:proper}
In the directed system of all regular modifications of $\Mbar_{g,n}$, those $X$ such that $\on{DRL}_{\ul a}^X \to \Mbar_{g,n}$ is proper form a cofinal system. 
\end{theorem}

Now $\on{DRC}_{\ul a}^X$ is supported on $\on{DRL}_{\ul a}^X$, so when the map $\on{DRL}_{\ul a}^X \to \Mbar_{g,n}$ is proper we can take the pushforward of $\on{DRC}_{\ul a}^X$ to $\Mbar_{g,n}$. Writing ${\pi_X}_*\on{DRC}_{\ul a}^X$ for the resulting cycle on $\Mbar_{g,n}$, we have:

\begin{theorem}[\cite{Holmes2017Extending-the-d}, theorem 1.2]\label{main_theorem:limit}
The net ${\pi_X}_*\on{DRC}_{\ul a}^X$ is eventually constant in the Chow ring  $\on{CH}_{\bb Q}(\Mbar_{g,n})$. We denote the limit by $\overline{\on{DRC}}_{\ul a}$. 
\end{theorem}

In the case $k=0$ it is shown in \cite{Holmes2017Extending-the-d} that this class $\overline{\on{DRC}}_{\ul a}$ coincides with the class constructed by Li, Graber, and Vakil. 


\section{Multiplicativity lemma}\label{sec:multiplicativity}
Let $S$ be a regular algebraic stack, and $G/S$ a smooth separated group scheme with unit section $e$. Given $\sigma \in G(S)$ a section, we define 
\begin{equation*}
L_\sigma = \sigma^*e
\end{equation*}
as a closed substack of $S$, and 
\begin{equation*}
C_\sigma = \sigma^*[e]
\end{equation*}
as a cycle class supported on $L_\sigma$. 
\begin{lemma}[Multiplicativity lemma]\label{lemma:mult}
Let $\pi\colon G\to S$ be as above, and let $\sigma$, $\tau \in G(S)$ be two sections. Then we have 
\begin{equation}\label{equality:locus}
L_\sigma \times_S L_\tau = L_\sigma \times_S L_{\sigma + \tau}
\end{equation}
as closed substacks of $S$, and 
\begin{equation}\label{equality:cycle}
C_\sigma \cdot C_\tau = C_\sigma \cdot C_{\sigma + \tau}
\end{equation}
as cycles supported on $L_\sigma \times_S L_\tau$. 
\end{lemma}
\begin{proof}
Note that the set-theoretic version of \ref{equality:locus} is trivial. We give only the argument for \ref{equality:cycle}; that for \ref{equality:locus} is similar but easier. In the diagram 
\begin{equation*}
\begin{tikzcd}
G \arrow[r, "i"] & G \times_S G \arrow[d] \arrow[r, "m"] & G \\
& S \arrow[u,bend left, "{(\sigma,\tau)}"] \arrow[ur, swap, "\sigma + \tau"] &\\
\end{tikzcd}
\end{equation*}
where $i = (e\circ \pi, id)$, we have equalities of cycles supported on $L_\sigma \cap L_\tau$:
\begin{equation*}
\begin{split}
\sigma^*[e] \cdot (\sigma + \tau)^*[e] & = (\sigma, \tau)^*i_*[G] \cdot (\sigma, \tau)^*(m^*[e])\\
 & = (\sigma, \tau)^*\Big(  i_*[G] \cdot m^*[e] \Big)\\
\text{ (projection formula)} & = (\sigma, \tau)^*i_*\Big( [G] \cdot i^*m^*[e] \Big) \\
 & = (\sigma, \tau)^*i_*i^*m^*[e]\\ 
 & = (\sigma, \tau)^*[(e,e)]\\
 & = \sigma^*[e] \cdot  \tau^*[e]. 
\end{split}
\end{equation*}
\end{proof}

A natural application of this lemma is to the double ramification cycle. Here the base $S$ is given by $\ca M_{g,n}$, and $G = \ca J_{g,n}$ is the jacobian of the universal curve. Then for any vector of integers $\ul a = (a_1, \dots, a_n, k)$ with $\sum_i a_i = k(2g-2)$ we have the section $\sigma_{\ul a} = [\omega^{\otimes k}\left(-\sum_i a_i x_i\right)]$ of $\ca J_{g,n}$, and the double ramification cycle on $\ca M_{g,n}$ is given by pulling back the unit section along $\sigma_{\ul a}$, i.e. 
\begin{equation*}
\on{DRC}_{\ul a} = C_{\sigma_{\ul a}}
\end{equation*}
in the notation of \ref{lemma:mult}. We thus obtain from \ref{lemma:mult} the relation 
\begin{equation}\label{eq:DRC_relation}
\on{DRC}_{\ul a} \cdot \on{DRC}_{\ul b} = \on{DRC}_{\ul a} \cdot \on{DRC}_{\ul a+ \ul b}
\end{equation}
 in $\on{CH}_{\bb Q}(\ca M_{g,n})$, after pushing forward from the intersection of the corresponding double ramification loci. However, this relation is uninteresting as both sides vanish, since the degree $2g$-part of the tautological ring vanishes here. 
 
Over the locus of compact type (or more generally \emph{treelike}) curves, the double ramification cycle can be defined in the same way, and the same proof shows that multiplicativity holds here; more details are given in \ref{sec:treelike}. Moreover, on these loci the relation is not vacuous, as shown in \ref{sec:example}. However, the same section shows that this multiplicativity relation does not extend over the whole of $\Mbar_{g,n}$; in the next section, we introduce the b-Chow ring, and in the section after we extend the double ramification cycle to the b-Chow ring and show that multiplicativity does hold there.

\section{The $b$-Chow ring}\label{sec:b_Chow}

The group of $b$-divisors on a scheme $X$ was introduced by Shokurov \cite{Shokurov19963-fold-log-mode}, \cite{Shokurov2003Prelimiting-fli} as the limit of the divisor groups of all blowups of $X$, with transition maps given by proper pushforward. One can define a (large) $b$-Chow group in the same way, as the limit over all blowups with transition maps given by pushforward, but note that it does not have a natural ring structure. The small $b$-Chow group is defined below as the \emph{colimit} of Chow groups over smooth blowups, with transition maps given by pullback of cycles. It is naturally a subgroup of the large $b$-Chow group, and importantly it carries a natural ring structure (described below), so we refer to it as the (small) $b$-Chow \emph{ring}. 

Let $S$ be an irreducible noetherian algebraic stack. We write $\on{Bl}(S)$ for the category whose objects are proper birational maps $X \to S$, relatively representable by algebraic spaces, and with $X$ regular, and where the morphisms are morphisms over $S$. Taking Chow rings and pullbacks gives a new category $\on{CH}_{\bb Q}(\on{Bl}(S))$, whose objects are the $\bb Q$-Chow rings of the objects of $\on{Bl}(S)$, and where morphisms are given by pullbacks (which makes sense because everything is regular). We define the $b$-Chow ring of $S$ to be the colimit of this system of rings: 
\begin{equation*}
\on{bCH}_{\bb Q}(S) = \on{colim}\on{CH}_{\bb Q}(\on{Bl}(S)). 
\end{equation*}
Since the category $\on{CH}_{\bb Q}(\on{Bl}(S))$ is filtered (\cite[\href{http://stacks.math.columbia.edu/tag/04AX}{Tag 04AX}]{stacks-project}) we can give a much more concrete description on the level of sets: 
\begin{equation*}
\on{bCH}_{\bb Q}(S) = \left(\bigsqcup_{X \in \on{Bl}(S)} \on{CH}_{\bb Q}(X) \right)/\sim
\end{equation*}
where for elements $x \in  \on{CH}_{\bb Q}(X)$ and $y \in  \on{CH}_{\bb Q}(Y)$, we say $x \sim y$ if and only if there exists $Z \in \on{Bl}(S)$ and maps $f\colon Z \to X$, $g \colon Z \to Y$, with 
\begin{equation*}
f^*x = g^*y. 
\end{equation*}
To multiply elements $x$ and $y$, we again find a $Z \in \on{Bl}(S)$ mapping to both $X$ and $Y$, and form the intersection product after pullback to this $Z$. 

\section{Multiplicativity of the double ramification cycle in the $b$-Chow ring}\label{sec:b_DRC}

Given $\ul a = (a_1, \dots, a_n, k)$ with $\sum_i a_i = k(2g-2)$, we first define the extension of the corresponding double ramification cycle to $\on{bDRC}_{\ul a}$ in $\on{CH}_{\bb Q}(\Mbar_{g,n})$. Taking the standard extension to the Chow ring of $\Mbar_{g,n}$ and pulling back is not the right approach --- for example, the multiplicativity relation will fail. Instead we look at the construction in \ref{sec:extending_DRC}. Recall that for a modification $X \to \Mbar_{g,n}$ we write $\mathring X \hra X$ for the largest open to which $\sigma_{\ul a}$ extends. We define $\on{DRL}_{\ul a}^X \tra \mathring X$ by pulling back the unit section scheme-theoretically, and $\on{DRC}_{\ul a}^X$ as a cycle class on $\on{DRL}_{\ul a}^X$ by pulling back in Chow. 

Let $X$ be regular and such that $\on{DRL}_{\ul a}^X$ is proper over $\Mbar_{g,n}$. Write $i\colon \on{DRL}_{\ul a}^X \to X$ for the inclusion, which is a closed immersion. Then we define $\on{DRC}^X_{\ul a} = i_*\on{DRC}_{\ul a}^X$ as an element of $\on{bCH}_{\bb Q}(\Mbar_{g,n})$. Recall from \ref{sec:extending_DRC} that the $X$ with $\on{DRL}_{\ul a}^X$ proper over $\Mbar_{g,n}$ form a cofinal system among all modifications $X$, yielding a net of $\on{DRC}^X_{\ul a}$ in $\on{bCH}_{\bb Q}(\Mbar_{g,n})$. 

\begin{lemma}
The net of $\on{DRC}^X_{\ul a}$ in $\on{bCH}_{\bb Q}(\Mbar_{g,n})$ is eventually constant. 
\end{lemma}
\begin{proof}
This argument is a simpler version of the proof of \cite[theorem 6.3]{Holmes2017Extending-the-d}. The limiting value can be obtained by taking $X$ a regular compactification of the stack $\ca M_{g,n}^\lozenge$ constructed in [loc.cit.]. 
\end{proof}

\begin{definition}
We define $\on{bDRC}_{\ul a}$ in $\on{bCH}_{\bb Q}(\Mbar_{g,n})$ as the limit of the above net. 
\end{definition}

\Cref{thm:pushforward} now follows formally from \cite[\S 6]{Holmes2017Extending-the-d}, so it remains to prove \ref{thm:mult_intro}. 

\begin{theorem}\label{thm:main}
Choose $\ul a = (a_1, \dots, a_n, k)$ with $\sum_i a_i = k(2g-2)$, and similarly choose $\ul b = (b_1, \dots, b_n, k')$. Then in $\on{bCH}_{\bb Q}(\Mbar_{g,n})$ we have 
\begin{equation}\label{eq:rel_thm}
\on{bDRC}_{\ul a} \cdot \on{bDRC}_{\ul b} = \on{bDRC}_{\ul a} \cdot \on{bDRC}_{\ul a+ \ul b}. 
\end{equation}
\end{theorem}
\begin{proof}
Choose $X \to \Mbar_{g,n}$ so that $\on{DRL}_{\ul  a}^X \to \Mbar_{g,n}$ is proper and so that $\on{DRC}_{\ul a}^X$ equals the limiting value $\on{bDRC}_{\ul a}$ in $\on{bCH}_{\bb Q}(\Mbar_{g,n})$. Choose a corresponding $Y$ for $\ul b$, and let $Z$ be a regular modification admitting maps to $X$ and $Y$ over $\Mbar_{g,n}$. It suffices to check \ref{eq:rel_thm} in the Chow ring of $Z$. 

Let $\mathring Z_a \hra Z$ be the largest open where $\sigma_{\ul a}$ extends, and similarly define $\mathring Z_{\ul b}$ and $\mathring Z_{\ul a + \ul b}$. Writing $\mathring Z = \mathring Z_{\ul a}  \cap \mathring Z_{\ul b}$, we see that $\sigma_{\ul a + \ul b}$ is also defined on $\mathring Z$; it is given by $\sigma_{\ul a} + \sigma_{\ul b}$. Hence we have 
\begin{equation*}
\mathring Z = \mathring Z_{\ul a}  \cap \mathring Z_{\ul b} \sub \mathring Z_{\ul a + \ul b}, 
\end{equation*}
and a similar argument shows 
\begin{equation}\label{eq:inclusion}
\mathring Z_{\ul a}  \cap \mathring Z_{\ul a + \ul b} \sub \mathring Z_{\ul b}. 
\end{equation}
Now it is clear that
\begin{equation*}
\on{DRL}^Z_{\ul a} \cap \on{DRL}^Z_{\ul b} \sub  \mathring Z_{\ul a}  \cap \mathring Z_{\ul b} = \mathring Z, 
\end{equation*}
and similarly we have
\begin{equation*}
\on{DRL}^Z_{\ul a} \cap \on{DRL}^Z_{\ul a + \ul b} \sub  \mathring Z_{\ul a}  \cap \mathring Z_{\ul a + \ul b}  \sub \mathring Z_{\ul a}  \cap \mathring Z_{\ul b} = \mathring Z, 
\end{equation*}
where the middle relation comes from \ref{eq:inclusion}. The theorem now follows directly from \ref{lemma:mult} applied to the universal jacobian $\ca J_{g,n}$ pulled back to $\mathring Z$. 
\end{proof}

\section{Relation to the cycle ${\operatorname P}_g^{d,k}({A})$}
\label{sec:relation_to_P}
In this and the next section we consider the connection between the classes $\overline{\on{DRC}}_{\ul a} \in \on{CH}_{\bb Q}^g(\Mbar_{g,n})$ and the tautological cycle class $\mathrm{P}_g^{d,k}({\ul A})$ introduced by the second author in \cite{Janda2016Double-ramifica}. For this we first recall some notation from [loc.cit.].

Fix an integer $k \geq 0$ and an integer vector ${\ul A}=(A_1, \ldots, A_n)$ with $\sum_{i=1}^n A_i = k (2g-2+n)$. Note that there is a natural bijection of such vectors ${\ul A}$ and vectors ${\ul a}=(a_1, \ldots, a_n)$ with $\sum_{i=1}^n a_i = k (2g-2)$ by setting
\[(A_1, \ldots, A_n) = (a_1 +k , \ldots, a_n+k);\]
we will use this identification in what follows.

Fix also a degree $d \geq 0$, then given this data, in \cite[Section 1.1]{Janda2016Double-ramifica} a tautological cycle class 
\[\mathrm{P}_g^{d,k}({\ul A}) \in \on{CH}_{\bb Q}^d(\Mbar_{g,n})\]
is defined as an explicit sum in terms of decorated boundary strata. The main result of \cite{Janda2016Double-ramifica} is that for $k=0, d=g$ this formula computes the double ramification cycle corresponding to the partition ${\ul A}$. More precisely, they prove
\[\mathrm{DR}_g({\ul A}) = 2^{-g} \mathrm{P}_g^{g,0}({\ul A}),\]
where $\mathrm{DR}_g({\ul A})$ is the double ramification cycle associated to $\ul A$ via the Gromov-Witten theory of `rubber $\bb P^1$'. 

From \cite[conjecture 1.4]{Holmes2017Extending-the-d} we recall 
\begin{conjecture}\label{conj:pixton_eq}
For all $k$ we have 
\begin{equation*}
\overline{\on{DRC}}_{\ul a} =  2^{-g}\mathrm{P}_g^{g,k}({\ul A})
\end{equation*}
as elements of $\on{CH}_{\bb Q}^g(\Mbar_{g,n})$.
\end{conjecture}
\begin{remark}\label{prop:pixton_eq_k_zero}
\Cref{conj:pixton_eq} holds when $k=0$. Indeed, when $k=0$ we know by \cite[theorem 1.3]{Holmes2017Extending-the-d} that $\overline{\on{DRC}}_{\ul a} = \mathrm{DR}_g({\ul A})$, which combined with the main result of \cite{Janda2016Double-ramifica} yields the result. 
\end{remark}
We now show that \ref{conj:pixton_eq} holds in cohomology for all $k$ if we restrict to the locus of curves of compact type. 
\begin{proposition}\label{prop:pixton_eq_ct} \label{pro:comparison}
On the locus $\ca M_{g,n}^{ct}$ of compact type curves we have an equality
  \begin{equation} \label{eqn:DRC_Pixton}
  \overline{\on{DRC}}_{\ul a}  = 2^{-g}\mathrm{P}_g^{g,k}({\ul A})  \in \on{H}_{\bb Q}^{2g}(\ca M_{g,n}^{ct}).
 \end{equation}
\end{proposition}
\begin{proof}
The proof runs via the following chain of equalities in $\on{H}_{\bb Q}^{2g}(\ca M_{g,n}^{ct})$. 
 \begin{equation*}
  \overline{\on{DRC}}_{\ul a} \stackrel{a)}{=} \sigma_{\ul a}^* [e] \stackrel{b)}{=} \sigma_{\ul a}^* \frac{\theta^g}{g!} = \frac{1}{g!} \left( \sigma_{\ul a}^* \theta\right)^g \stackrel{c)}{=} \frac{1}{2^gg!} \left( P_g^{1,k}({\ul A}) \right)^g \stackrel{d)}{=} 2^{-g}\mathrm{P}_g^{g,k}({\ul A}).
 \end{equation*}
 We expect all of these equalities to hold in Chow, but for c) we only know it in cohomology. To start, equality a) follows from the definition of the double ramification cycle and the fact that the universal jacobian over $\ca M_{g,n}^{ct}$ is an abelian scheme, hence any sections over $\ca M_{g,n}$ are guaranteed to extend (uniquely) over $\ca M_{g,n}^{ct}$. Equality b) comes by pulling back the obvious relation on the universal abelian variety, which has already been observed by various authors, see for instance \cite{Grushevsky2012The-double-rami}. 

 Now the pullback (in cohomology) of the theta divisor under $\sigma_a$ has been computed by Hain in \cite{Hain2013Normal-function}. In standard notation for the tautological classes in $\ca M_{g,n}^{ct}$, Hain's result reads as follows: in $H^2(\mathcal{M}_{g,n}^{\mathrm{ct}})$ we have
 \begin{equation} \label{eqn:Hain} \sigma_{\ul a}^* \theta = -\frac{k^2}{2} \kappa_1 + \frac{1}{2}\sum_{j=1}^n (a_j+k)^2 \psi_j - \frac{1}{2} \sum_{g',P} (a_P - (2g'-1)k)^2 \delta_{g'}^P. \end{equation}
Here $P$ runs over subsets of $\{1,\ldots,n\}$, $a_P = \sum_{i \in P} a_i$, and the last sum should be interpreted as including each boundary divisor $\delta_{g'}^P=\delta_{g-g'}^{P^c}$ exactly once. Deducing equality c) then follows by an elementary verification using the definition of $\mathrm{P}_g^{1,k}({\ul A})$ from \cite[Section 1.1]{Janda2016Double-ramifica}. 

Finally, equality d) follows from the fact that on $\ca M_{g,n}^{ct}$ we have an equality of mixed-degree classes
 \begin{equation} \label{eqn:expPixton}
  \exp(\mathrm{P}_g^{1,k}({\ul A})) = \sum_{d \geq 0} \mathrm{P}_g^{d,k}({\ul A}) \in \on{CH}_{\bb Q}^*(\ca M_{g,n}^{ct}).
 \end{equation}
 This equality is a combinatorial statement; it is a specialization of the more general \ref{lemma:pixton_tl} which we will prove in the next section.
\end{proof}

\section{Restricting to treelike curves}\label{sec:treelike}

In this section we focus on the locus of \emph{treelike curves} --- these are stable curves whose graph is a tree with any number of self-loops attached (equivalently, all non-disconnecting edges are self-loops, or all cycles in the graph have length $\le 1$). We write $\ca M_{g,n}^{tl}$ for this locus; it is open in $\Mbar_{g,n}$, and clearly contains the compact-type locus $\ca M_{g,n}^{ct}$. 

Over the locus of treelike curves, the universal jacobian $\ca J_{g,n}$ is not proper (and its toric rank can be arbitrarily large). However, it is still a N\'eron model of its generic fibre in the sense of \cite{Holmes2014Neron-models-an} --- this follows easily from the main theorem of [loc.cit.], since all cycles in the graph have length at most 1. In particular, this implies that the section $\sigma_{\ul a}$ over $\ca M_{g,n}$ extends uniquely to $\ca J_{g,n}$ over the whole of $\ca M_{g,n}^{tl}$ (in fact, $\ca M_{g,n}^{tl}$ can be uniquely characterised as the largest open of $\Mbar_{g,n}$ such that every section of the universal jacobian over $\ca M_{g,n}$ extends). 

Recall that $\overline{\on{DRC}}_{\ul a}$ is the extension of the double ramification cycle to the Chow ring of $\Mbar_{g,n}$ as constructed in \ref{sec:extending_DRC}. 
\begin{lemma}\label{lem:tl} 
In $\on{CH}_{\bb Q}^g(\ca M_{g,n}^{tl})$ we have the equality
\begin{equation}\label{eqn:DRC_pullback}
\overline{\on{DRC}}_{\ul a} = \sigma_{\ul a}^*[e]
\end{equation}
and in $\on{CH}_{\bb Q}^{2g}(\ca M_{g,n}^{tl})$ we have
\begin{equation}\label{eq:tl_mult}
\overline{\on{DRC}}_{\ul a} \cdot \overline{\on{DRC}}_{\ul b} = \overline{\on{DRC}}_{\ul a} \cdot \overline{\on{DRC}}_{\ul a + \ul b}. 
\end{equation}
\end{lemma}
\begin{proof}
Since the section $\sigma_{\ul a}$ extends to $\ca J_{g,n}$ over $\ca M_{g,n}^{tl}$, the blowups used to extend the section may be assumed to be isomorphisms over $\ca M_{g,n}^{tl}$; more precisely, the cofinal system in \ref{sec:b_DRC} can be chosen so that all of the birational maps $X \to \Mbar_{g,n}$ are isomorphisms over $\ca M_{g,n}^{tl}$. This proves (\oref{eqn:DRC_pullback}), and (\oref{eq:tl_mult}) then follows from \ref{thm:main}, or directly from \ref{lemma:mult}. 
\end{proof}

The classes $\mathrm{P}_g^{g,k}$ also satisfy multiplicativity on $\ca M_{g,n}^{tl}$:
\begin{proposition}\label{prop:pixton_mult}
Let $\ul A$, $\ul B$ be vectors of $n$ integers with $\sum A_i = k_{\ul a} (2g-2+n)$ and $\sum B_i = k_{\ul b} (2g-2+n)$ for some $k_{\ul a}, k_{\ul b} \in \mathbb{Z}$. Then the equality
 \begin{equation}  \label{eqn:multPixton}
  \mathrm{P}_g^{g,k_{\ul a}}(\ul A ) \cdot \mathrm{P}_g^{g,k_{\ul b}}(\ul B ) = \mathrm{P}_g^{g,k_{\ul a}}(\ul A ) \cdot \mathrm{P}_g^{g,k_{\ul a} + k_{\ul b}}(\ul A + \ul B )
 \end{equation}
holds in $\on{CH}_{\bb Q}^{2g}(\ca M_{g,n}^{tl})$. 
\end{proposition}
The consistency of this multiplicativity with that of \ref{lem:tl} provides evidence for \ref{conj:pixton_eq}.

There are two key ingredients in the proof of \ref{prop:pixton_mult}. The first is the basic codimension $g+1$ relation
\begin{equation} \label{eqn:g+1vanishing}
  \mathrm{P}_g^{g+1,k}(\ul A ) = 0 \in \on{CH}_{\bb Q}^{g+1}(\Mbar_{g,n})
 \end{equation}
 proved in \cite[Theorem 5.4]{CladerJanda2016}. The second is the following combinatorial lemma:
 \begin{lemma} \label{lemma:pixton_tl}
 Let $\mathrm{P}_g(\ul A)^{tl}$ denote the mixed-degree class in the Chow ring of the locus of treelike curves
 \[
 \mathrm{P}_g(\ul A)^{tl} := \sum_{d \geq 0} \mathrm{P}_g^{d,k}({\ul A}) \in \on{CH}_{\bb Q}^*(\ca M_{g,n}^{tl}).
 \]
 Then there exists a mixed-degree class $\Delta\in \on{CH}_{\bb Q}^*(\ca M_{g,n}^{tl})$ (not depending on ${\ul A}$) along with a divisor-valued quadratic form $Q({\ul A})\in \on{CH}_{\bb Q}^1(\ca M_{g,n}^{tl})$ such that
 \[
 \mathrm{P}_g(\ul A)^{tl} = \exp(Q({\ul A}))\Delta.
 \]
 \end{lemma}
 
 Before checking \ref{lemma:pixton_tl}, we use it to prove \ref{prop:pixton_mult}:
 \begin{proof}[Proof of \ref{prop:pixton_mult}]
 Using \ref{lemma:pixton_tl} we can rewrite the codimension $g+1$ relation for a vector ${\ul A} + {\ul C}$ as
 \[
 [\exp(Q({\ul A} + {\ul C}))\Delta]_{g+1} = 0,
 \]
 where $[X]_d$ denotes the codimension $d$ part of a mixed-degree class $X$.
 
 This relation is an equality of polynomials in the ${\ul A}$ and ${\ul C}$ variables, so it will still hold if we restrict to the part of degree $1$ in ${\ul C}$. This gives
 \[
 [\exp(Q({\ul A}))\Delta]_{g}\cdot(Q({\ul A} + {\ul C}) - Q({\ul A}) - Q({\ul C})) = 0.
 \]
 Changing variables with ${\ul C} = {\ul A} + 2{\ul B}$, using the fact that $Q$ is a quadratic form, and dividing by $2$, we arrive at the relation
 \[
 [\exp(Q({\ul A}))\Delta]_{g}\cdot(Q({\ul A} + {\ul B}) - Q({\ul B})) = 0.
 \]
 
 Now, the mixed-degree class $\exp(Q({\ul A} + {\ul B})) - \exp(Q({\ul B}))$ is clearly divisible by the divisor class $Q({\ul A} + {\ul B}) - Q({\ul B})$, so we have the relation
 \[
 [\exp(Q({\ul A}))\Delta]_{g}[(\exp(Q({\ul A} + {\ul B})) - \exp(Q({\ul B})))\Delta]_g = 0.
 \]
 Then applying \ref{lemma:pixton_tl} again gives the desired multiplicativity statement.
 \end{proof}

\begin{proof}[Proof of \ref{lemma:pixton_tl}]
This lemma is essentially a combinatorial statement about the definition of the classes $\mathrm{P}_g^{d,k}({\ul A})$ in \cite[Section 1.1]{Janda2016Double-ramifica} along with the multiplication formula for tautological classes given in \cite[Appendix A, eq. (11)]{GraberPandharipande2003}. In the general case, $\mathrm{P}_g^{d,k}({\ul A})$ is a sum over decorated (by $\psi$ and $\kappa$ classes) dual graphs $\Gamma$ of a combinatorial coefficient times the tautological class corresponding to $\Gamma$. The combinatorial coefficient is defined by taking the $r$-constant term of a polynomial in $r$ defined by summing over certain balanced `weightings mod $r$' of the half-edges of $\Gamma$.

In our case, we can assume that the graph $\Gamma$ is treelike and the combinatorial coefficients then become significantly simpler: the only weights that are allowed to vary are those in loops of the graph. The result is that the coefficient associated to a graph $\Gamma$ factors as a product of the contributions from the loops and the contributions from the non-loops. Using the graph refinement calculus of the tautological ring multiplication formula \cite[Appendix A]{GraberPandharipande2003}, this means that the entire mixed-degree class factors:
\[
\mathrm{P}_g({\ul A})^\text{treelike} = \mathrm{P}_g({\ul A})^\text{tree}\cdot \mathrm{P}_g({\ul A})^\text{irred},
\]
where the three classes are, respectively, the full class on the locus of treelike curves, those terms with $\Gamma$ a tree, and those terms where $\Gamma$ has exactly one vertex with no $\kappa$ decorations on it and no $\psi$ decorations on any legs (but possibly on loops). Moreover, the final class does not actually depend on the vector ${\ul A}$; we set
\[
\Delta := \mathrm{P}_g({\ul A})^\text{irred}
\]
(an explicit formula for $\Delta$ in terms of Bernoulli numbers can easily be written down, but we have no need for it here). 

For the remaining factor $\mathrm{P}_g({\ul A})^\text{tree}$, we claim that
\begin{equation}\label{eqn:exp}
\mathrm{P}_g({\ul A})^\text{tree} = \exp([\mathrm{P}_g({\ul A})^\text{tree}]_{\text{deg $1$}}).
\end{equation}
Then we can take
\[
Q({\ul A}) := [\mathrm{P}_g({\ul A})^\text{tree}]_{\text{deg $1$}},
\]
which explicitly is given by the same formula as Hain's formula \ref{eqn:Hain} (multiplied by $2$ and interpreted as divisors on the locus of treelike curves) and thus is a quadratic form in ${\ul A}$.

It remains to check \ref{eqn:exp} using the multiplication formula of \cite[Appendix A]{GraberPandharipande2003}. Suppose that for $i=1,\ldots,k$, $\delta_{g_i}^{P_i}$ are boundary divisor classes for separating nodes, so each such class corresponds to a graph with two vertices connected by a single edge along with a distribution of the total genus $g$ and markings between the two vertices (such that one has genus $g_i$ and marking $P_i$). If we multiply all of these $k$ divisor classes together, the multiplication formula in this case says that the result is a sum over the following data: a tree $\Gamma$ along with a distribution of genus and markings between the vertices of $\Gamma$ and a sequence of edges $e_1,\ldots,e_k$ in $\Gamma$ (possibly with repetition) such that
\begin{enumerate}
    \item the division of genus and markings across the two sides of edge $e_i$ agree with the division in $\delta_{g_i}^{P_i}$;
    \item every edge of $\Gamma$ appears at least once in the sequence $e_1,\ldots,e_k$.
\end{enumerate}
Repeated edges $e_i$ give rise to $\psi$ classes along that edge. 

Computing the right side of \ref{eqn:exp} (the exponential of a divisor class) by using the above procedure to multiply divisor classes together then gives precisely the sum over trees appearing in the definition of $\mathrm{P}_g({\ul A})^\text{tree}$.

\end{proof}

\section{Failure of multiplicativity in the Chow ring of $\Mbar_{g,n}$}\label{sec:example}
Since both sides of \ref{eqn:multPixton} make sense in the Chow ring of $\Mbar_{g,n}$, it is natural to ask whether the multiplicativity stated in \ref{prop:pixton_mult} might hold not just on the locus of treelike curves but on the entire space of stable curves. In this section we present an explicit example where this desired equality fails and in fact argue that there can be no other extension of the cycles $\on{DRC}_{\ul a}$ from $\ca M_{g,n}^{ct}$ that would make the equality hold. In other words, multiplicativity is really a feature of the (small) $b$-Chow ring and not of the standard Chow ring.

 Let $g=1, k=0$ and consider the two partitions ${\ul a}=(2,4,-6)$, ${\ul b}=(-3,-1,4)$ of $0$. Let $\overline{\on{DRC}}_{\ul a}$, $\overline{\on{DRC}}_{\ul b}$, $\overline{\on{DRC}}_{{\ul a}+{\ul b}} \in \on{CH}_{\bb Q}^1(\Mbar_{1,3})$ be the corresponding double ramification cycles. 
 By  \ref{pro:comparison} these agree with the corresponding $\mathrm{P}_1^{1}({\ul A})$, which can be computed as explicit tautological classes.
 Using an implementation of the tautological ring by the second author one can check that the multiplicativity fails inside the Chow group of $\Mbar_{1,3}$, i.e.
 \begin{equation} \label{eq:relation_example}
  \overline{\on{DRC}}_{\ul a} \cdot \overline{\on{DRC}}_{\ul b} \neq \overline{\on{DRC}}_{\ul a} \cdot \overline{\on{DRC}}_{{\ul a}+{\ul b}} \in \on{CH}_{\bb Q}^2(\Mbar_{1,3}).
 \end{equation}
 What is true however is that the difference of the two sides in \ref{eq:relation_example} is a linear combination of the classes of the three irreducible components of $\Mbar_{1,3} \setminus \ca M_{1,3}^{tl}$. In other words, \ref{eq:relation_example} becomes an equality once we restrict to the locus $\ca M_{1,3}^{tl}$ of treelike curves, as proved in \ref{prop:pixton_mult}. Moreover, actually both sides of \ref{eq:relation_example} give nontrivial elements of $\on{CH}_{\bb Q}^2(\ca M_{1,3}^{tl})$. In particular, this shows that for the above example the two sides of the multiplicativity statement in the (small) $b$-Chow ring are also nontrivial. This gives an indication that the multiplicativity there does not hold for some trivial reason (like both sides always vanishing, for instance).
 
 Now one final hope for multiplicativity on $\Mbar_{g,n}$ could be that the cycles $\overline{\on{DRC}}_{\ul a}$, $\overline{\on{DRC}}_{\ul b}$, $\overline{\on{DRC}}_{{\ul a}+{\ul b}}$ are not the right extension of the corresponding Abel-Jacobi pullbacks $\sigma_{\ul a}^*[e], \sigma_{\ul b}^*[e], \sigma_{{\ul a}+{\ul b}}^*[e] \in \on{CH}_{\bb Q}^1(\ca M_{1,3}^{ct})$ on the locus of compact type curves. However, the complement $\Mbar_{1,3} \setminus \ca M_{g,n}^{ct}$ is exactly given by the boundary divisor $\Delta_{irr}$ generically parametrising irreducible nodal curves. Hence any such extensions must have the form
 \begin{align*}
  \widetilde{\on{DRC}}_{\ul a} &= \overline{\on{DRC}}_{\ul a} + \lambda_{\ul a} \cdot \Delta_{irr},\\
  \widetilde{\on{DRC}}_{\ul b} &= \overline{\on{DRC}}_{\ul b} + \lambda_{\ul b} \cdot \Delta_{irr},\\
  \widetilde{\on{DRC}}_{{\ul a}+{\ul b}} &= \overline{\on{DRC}}_{{\ul a}+{\ul b}} + \lambda_{{\ul a}+{\ul b}} \cdot \Delta_{irr}.
 \end{align*}
 Using that $\Delta_{irr}^2=0$ we compute 
 \begin{align*}
  &\widetilde{\on{DRC}}_{\ul a} \cdot (\widetilde{\on{DRC}}_{\ul b} - \widetilde{\on{DRC}}_{{\ul a}+{\ul b}})\\ 
  =& \underbrace{\overline{\on{DRC}}_{\ul a} \cdot (\overline{\on{DRC}}_{\ul b} - \overline{\on{DRC}}_{{\ul a}+{\ul b}})}_{I_1} + (\lambda_{\ul b} - \lambda_{{\ul a}+{\ul b}}) \underbrace{\overline{\on{DRC}}_{\ul a} \cdot \Delta_{irr}}_{I_2}+ \lambda_{\ul a} \underbrace{\Delta_{irr} \cdot (\overline{\on{DRC}}_{\ul b} - \overline{\on{DRC}}_{{\ul a}+{\ul b}})}_{I_3}.
 \end{align*} 
 However, it can be checked by computer that the three elements $I_1, I_2, I_3 \in \on{CH}_{\bb Q}^2(\Mbar_{1,3})$ are linearly independent. Therefore there is no way to choose $\lambda_{\ul a},\lambda_{\ul b},\lambda_{{\ul a}+{\ul b}}$ to have the $\widetilde{\on{DRC}}$ satisfy multiplicativity in the Chow ring of $\Mbar_{g,n}$; we only have multiplicativity in the (small) $b$-Chow ring or on the open locus of treelike curves.

\section{Logarithmic version; the Chow ring of the valuativisation}\label{sec:log_version}

We work with log structures in the sense of Fontaine-Illusie, using Olsson's generalisation to stacks \cite{Olsson2001Log-algebraic-s}. We put a log structure on $\Mbar_{g,n}$ and its universal curve as in Kato \cite{Kato1996Log-smooth-defo}. 

Following \cite{Kato1989Logarithmic-deg}, we define the \emph{valuativisation} of a log scheme or stack to be the limit of all the log blowups; this does not exist as a scheme (or stack), but does exist as either a locally ringed space or a pro-scheme over $\Mbar_{g,n}$; we take the latter approach. Taking a cofinal system of regular objects yields a natural Chow ring of the valuativisation as the colimit of the Chow rings of the blowups, just as in our section \ref{sec:b_Chow} but with a more restricted class of modifications. We can still define the double ramification cycle in this setting, since the modifications we take in \ref{sec:b_DRC} can be assumed to be logarithmic blowups, cf. \cite[Lemma 6.1]{Holmes2017Extending-the-d}. 

Why might we want to do this? Firstly, it reduces the number of modifications we have to work with (this can be a help or a hinderance, depending on circumstances). In particular, a logarithmic blowup can always be dominated by a regular logarithmic blowup, so this approach can be carried out in positive and mixed characteristic. In this way we can prove an analogue of \ref{thm:main} over $\on{Spec} \bb Z$, where the relation holds in the Chow ring of the valuativisation. 

Another reason to consider this approach is the derived equivalence between the valuativisation and a certain root stack, \cite{Scherotzke2016On-a-logarithmi}. We hope that this derived equivalence might shed some light on the relation between the first author's construction in \cite{Holmes2014A-Neron-model-o} of a universal N\'eron-model-admitting stack, and Chiodo's work \cite{Chiodo2015Neron-models-of}. More generally, it might realise our $\on{bDRC}\in \on{bCH}_{\bb Q}(\Mbar_{g,n})$ as a shadow of some more refined derived object.

\bibliographystyle{alpha} 
\bibliography{prebib.bib}

\end{document}